\documentclass{article}    

\usepackage{amsmath,amssymb,amsthm,epsfig}      

\newtheorem{theorem}{Theorem}[section]
\newtheorem*{theorem*}{Theorem}
\newtheorem{proposition}[theorem]{Proposition}

\newtheorem*{corollary*}{Corollary}

\newtheorem*{conjecture*}{Conjecture}

\theoremstyle{remark}

\theoremstyle{definition}

\renewcommand{\mod}{\mathrm{mod}}
\newcommand{\Zn}{\mathbb{Z}_{n}}

\newcommand{\hide}[1]{{}}  

\newcommand{\cliquer}{\texttt{cliquer}}

\begin{document}

\title{Improved Lower Bounds on the Classical Ramsey Numbers
$R(4,22)$ and $R(4,25)$}

\author{
Madison Lindsay$^1$ and John W. Cain$^2$
\footnote{$^1$ Department of Mathematics and Computer Science,
University of Richmond, 28 Westhampton Way, Richmond, Virginia 23173, USA.  
$^2$ Department of Mathematics, Harvard University, One Oxford Street, Cambridge, Massachusetts 02138, USA.
\newline
Corresponding author's e-mail: jcain2@math.harvard.edu}
}

\maketitle

\bigskip

{\small
\begin{center}
{\sc Abstract}
\end{center}
Circulant graphs have been used to effectively establish lower bounds on many classical Ramsey numbers.  Here, we
construct circulant graphs of prime order that sharpen the best published lower bounds on several Ramsey numbers.  Generalizing
previous work in which quadratic and cubic residues were used to construct circulant graphs for the same purpose, we report
on the use of quartic and higher-order residues.

\vspace{0.5 true cm}
}

\bigskip

\large

\section{Introduction}

The problem of calculating [classical, two-color] Ramsey numbers $R(p,q)$ presents an attractive
challenge for scholars with varied mathematical interests, including number theory, combinatorial design,
and high-performance computing.
Given positive integers $p$ and $q$, the {\em Ramsey number} $R(p,q)$ is defined as the least integer $n$
for which every graph on $n$ vertices contains either a clique of size $p$ or an independent set of size $q$
(see Graham, Rothschild and Spencer~\cite{graham-RS}, Section 4.1).
Equivalently, one writes $R(p,q) = n$ if
\begin{itemize}
  \item There exists a two-coloring (say red and blue) of the edges of the complete graph  on $n-1$ vertices (denoted $K_{n-1}$)
  containing {\em neither}
  a red monochromatic $K_{p}$ as a subgraph {\em nor} a blue monochromatic $K_{q}$ as a subgraph; and
  \item For every two-coloring of the edges of $K_{n}$, there exists {\em either} a red monochromatic $K_{p}$ as a subgraph {\em or}
  a blue monochromatic $K_{q}$ as a subgraph (and possibly both).
\end{itemize}
For a regularly-updated compendium of known results on Ramsey numbers, please refer to~\cite{rad}.

Many of the best known lower bounds on Ramsey numbers are constructive.
The construction shown in Figure~\ref{figure1} establishes that
$R(4,3) > 8$ by decomposing $K_{8}$ into two disjoint subgraphs, one containing no $K_{3}$ as a subgraph (left panel)
and the other containing no $K_{4}$ as a subgraph (right panel).
The figure illustrates one of many instances in which {\em circulant} graphs are used to establish lower bounds, and our main
result (Proposition 2.1 below) is another instance.
For that reason, let us define what it means for a graph $G$ to be circulant, formulating our definition so as to allow convenient notation
for such graphs.
Let $G$ be a graph on $n$ vertices whose vertex set is identified with $\Zn$.
We say that $G$ is {\em circulant} if there exists a vertex labeling and set $S \subset \left\{ 1, 2, \dots, \lfloor
\frac{n}{2} \rfloor \right\}$ such that $\{i,j\}$ is an edge if and only if
$$
\min \left\{ (i-j) \: (\mod \; n), \; \; (j-i) \: (\mod \; n) \right\} \in S.
$$
\noindent
For such graphs, we adopt the notation of Wu et al.~\cite{wu-SLX} and write $G = G_{n}(S)$.
The circulant graphs in the left and right panels of Figure~\ref{figure1} may be denoted $G_{8}(S_{1})$ and $G_{8}(S_{2})$
where $S_{1} = \left\{1, 4 \right\}$ and $S_{2} = \left\{2, 3\right\}$.



\begin{figure}
  \begin{center}
   \includegraphics[width=0.7\textwidth]{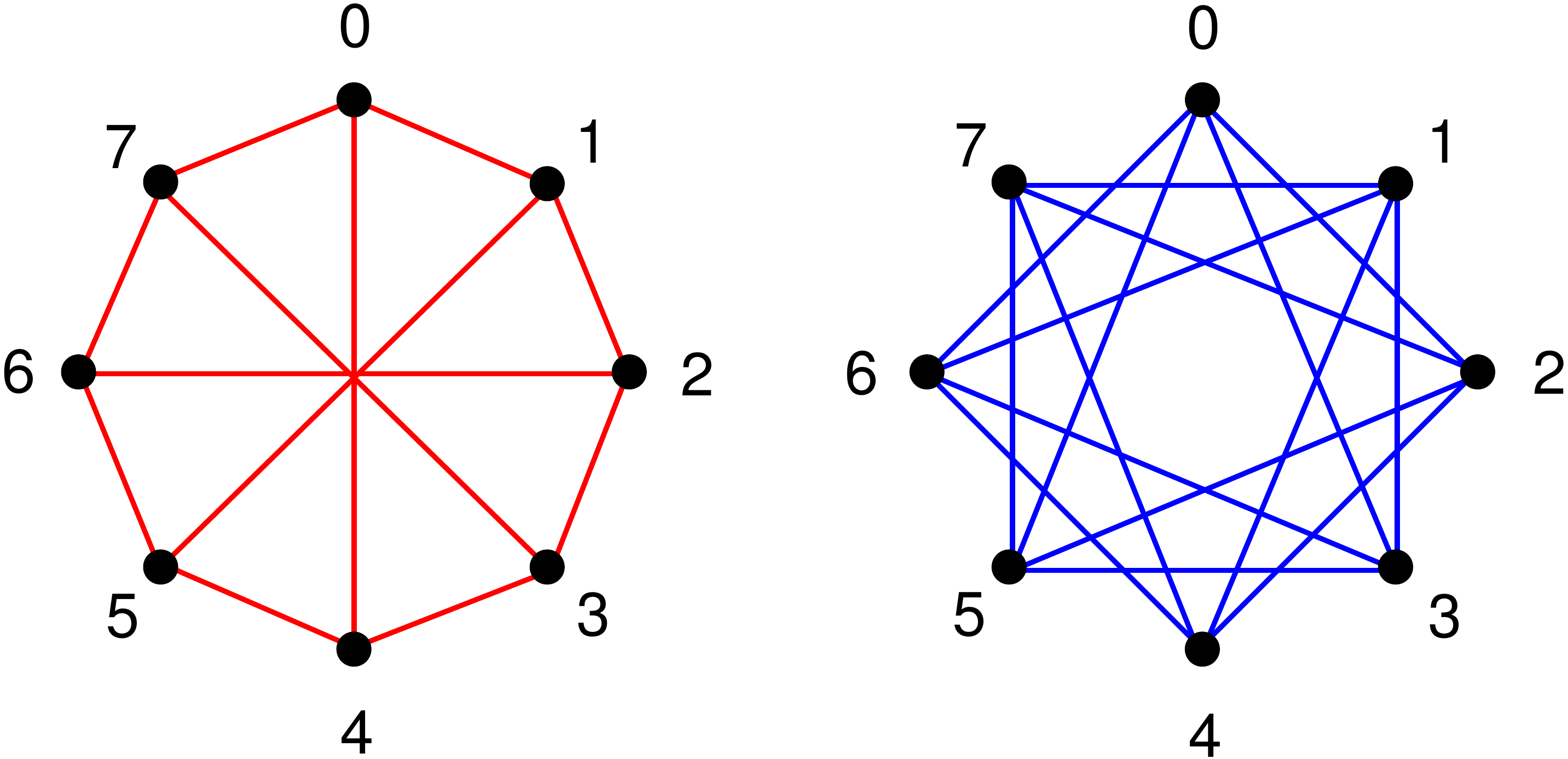}
  \end{center}
  \caption{A construction proving that $R(3,4) > 8$.}
  \label{figure1}
\end{figure}

\section{Using quartic residues to improve lower bounds}

Greenwood and Gleason~\cite{greenwood-G} suggested the use of quadratic, cubic, and quartic residues as a means of generating
circulant graphs which establish lower bounds on Ramsey numbers.  To show that $R(4,4) > 17$, one may partition the set
$S = \{ 1, 2, 3, \dots \frac{17-1}{2} \}$ into quadratic residues $S_{1} = \{1, 2, 4, 8\}$ and non-residues $S_{2} = \{ 3, 5, 6, 7 \}$
modulo 17.  The circulant graphs $G_{17}(S_{1})$ and $G_{2}(S_{2})$ form a partition of $K_{17}$ into disjoint subgraphs, neither of which
contains a copy of $K_{4}$ as a subgraph.  Similarly, to show that $R(6,6) > 101$  (the best known lower bound on that Ramsey number)
one may partition the set $S = \{ 1, 2, 3, \dots 50 \}$ into a disjoint union of sets $S_{1}$ and $S_{2}$ consisting of quadratic residues
and non-residues modulo 101, respectively, and argue that neither of the circulant graphs $G_{101}(S_{1})$ and $G_{101}(S_{2})$ contains $K_{6}$ as a
subgraph.
More recently, Su et al.~\cite{su-LLL} utilized cubic residues in a similar fashion, sharpening the best known lower bounds on 
multiple Ramsey numbers.

The proof of the following Proposition involves the use of quartic residues in precisely the same way.  {\em The two bounds in the
Proposition are the best published bounds, but we have reason to believe that others~\cite{lslw} may have discovered these bounds independently
several years ago, likely using the same construction that we offer.}


\begin{proposition}
The following lower bounds hold:
$$
\begin{aligned}
\mbox{(a)} & \qquad R(4,22) > 313; \\
\mbox{(b)} & \qquad R(4,25) > 457.
\end{aligned}
$$
\end{proposition}
\begin{proof}  
Note that 313 and 457 are primes, both of which are congruent to 1 (mod 8).

\vspace{0.3 true cm}
\noindent
(a)  For $p = 313$, define the following subsets of $S = \left\{ 1, 2, \dots, \frac{p-1}{2} \right\}$:
$$
\begin{aligned}
S_{1} & = \{ 1, 3, 4, 9, 11, 12, 16, 19, 26, 27, 33, 36, 44, 48, 50, 57, 58, 64, 70, 76, 78, 79, 81,\\
      & \qquad  83, 85, 98, 99, 103, 104, 108, 113, 119, 121, 132, 137, 139, 142, 144, 150 \} \\
S_{2} & = S \backslash S_{1}.
\end{aligned}
$$
\noindent
Here, $S_{1}$ consists of quartic residues $(\mod \; p)$ while $S_{2}$ consists of quartic nonresidues $(\mod \; p)$.
The circulant graphs $G_{p}(S_{1})$ and $G_{p}(S_{2})$ form a decomposition of the complete graph $K_{p}$ into disjoint
subgraphs.  Using the freely available software \cliquer~\cite{cliquer}, we calculated that the clique numbers $G_{p}(S_{1})$ 
and $G_{p}(S_{2})$ are $3$ and $21$, respectively, thereby establishing Part (a) of the Proposition.

\vspace{0.3 true cm}
\noindent
(b)  For $p = 457$ proceed as in the proof of Part (a), except using
$$
\begin{aligned}
S_{1} & = \{  1, 4, 6, 7, 9, 16, 17, 19, 24, 28, 29, 36, 42, 49, 50, 54, 63, 64, 68, 73, 75, 76, 79, 81, 94, \\
      & \qquad 96, 102, 107, 110, 112, 114, 116, 119, 130, 133, 134, 141, 144, 153, 155, 157, 163, \\
      & \qquad 165, 168, 171, 174, 185, 195, 196, 200, 201, 203, 205, 215, 216, 218, 227\} \\
S_{2} & = S \backslash S_{1}.
\end{aligned}
$$
\noindent
Using \cliquer, we calculated that the clique numbers of $G_{p}(S_{1})$ and $G_{p}(S_{2})$ are $3$ and $24$, respectively,
establishing Part (b) of the Proposition.
\end{proof}

\section{Discussion}
Our strategy of using quartic residues to generate circulant graphs for the purpose of improving lower bounds on classical Ramsey
numbers is far from novel.
Quartic residues were first mentioned in this context by Greenwood and Gleason~\cite{greenwood-G}.
In the sixty years since their article was published, the advent of modern computing has allowed for the exploration of much larger
graph structures.
Our computer-assisted proof of Proposition 2.1 exploits an efficient algorithm (at least by current standards)
for computing clique numbers, implemented via the freely available \cliquer~software.
Details regarding the algorithm used by \cliquer~are supplied in the documentation on the authors' website~\cite{cliquer}.
Of course, the problem of calculating clique numbers becomes computationally intractable for graphs of sufficient size.  Our simulations
were performed on a dedicated desktop computer running the Ubuntu Linux operating system.  Demonstrating that
$R(4,22) > 313$ required 7 hours of computing time, and verifying that $R(4,25) > 457$ required approximately 10 days.
We did not attempt computations that appeared likely to require more than two weeks of computing time on a 
standard desktop computer.
We did, however, explore the use of quartic, quintic, and higher-order (up to 8th power) residues to generate circulant graph
structures on at most 500 vertices.
Some of those graphs generated respectable lower bounds; for instance, quintic residues modulo 71 can be used to show that
$R(3,15) > 71$, which is nearly as strong as the best reported result of $R(3,15) > 72$.
However, among the graphs that we explored, quintic and higher residues did not produce improvements over any best published bounds.

We attempted (unsuccessfully) to calculate the clique numbers of the graph structures described in the proof of Proposition
2.1 without computer assistance.  We would be most interested to see a non-computer-assisted proof of these results.



\begin{center}
  {\bf Acknowledgments}
\end{center}
The lead author gratefully acknowledges the financial support of a University of Richmond Summer Research Fellowship.


\end{document}